\documentclass[a4paper,12pt,reqno]{amsart}

\usepackage{amsmath,amssymb,amsthm,amsaddr}
\usepackage[margin=3.4cm,top=3.5cm,bottom=3.5cm,footskip=1cm,headsep=1.5cm]{geometry}
\usepackage[utf8]{inputenc}

\usepackage[english]{babel}

\usepackage[font=footnotesize,labelfont=bf]{caption} 
\usepackage[noend]{algpseudocode}
\usepackage{algorithm}
\usepackage{xcolor}

\def\RR{\mathbb{R}}

\def\div{\operatorname{div}}
\def\curl{\operatorname{curl}}
\usepackage{graphicx}
\usepackage{xcolor}
\def\T{\mathcal{T}}
\def\energy{w}
\def\denergy{\partial_b w}
\def\ddenergy{\partial_{bb} w}
\def\linep{\tau}
\usepackage{subcaption}
\usepackage{float}
\usepackage{graphicx}

\def\Curl{\operatorname{Curl}}

\def\dy{\partial_y}

\def\dx{\partial_{x}}

\def\lines{\psi}

\def\Th{\mathcal{T}_h}

\def\N{\mathcal{N}}
\def\RT{\mathcal{RT}}

\newtheorem{lemma}{Lemma}
\newtheorem{theorem}[lemma]{Theorem}

\theoremstyle{definition}
\newtheorem{assumption}[lemma]{Assumption}

\newtheorem{remark}[lemma]{Remark}
\usepackage{url}

\begin{document}
\title[Higher order FEM for nonlinear magnetostatics]{On the convergence of higher order finite element methods for nonlinear magnetostatics}

\author{H. Egger$^{1,2}$, F. Engertsberger$^1$, and B. Radu$^2$}

\address{%
\small 
$^1$Institute of Numerical Mathematics, Johannes Kepler University Linz, Austria \\
$^2$Johann Radon Institute for Computational and Applied Mathematics, Linz, Austria}
%

\begin{abstract}
The modeling of electric machines and power transformers typically involves systems of nonlinear magnetostatics or -quasistatics, and their efficient and accurate simulation is required for the reliable design, control, and optimization of such devices. 
We study the numerical solution of the vector potential formulation of nonlinear magnetostatics by means of higher-order finite element methods. 
Numerical quadrature is used for the efficient handling of the nonlinearities and domain mappings are employed for the consideration of curved boundaries.
The existence of a unique solution is proven on the continuous and discrete level and a full convergence analysis of the resulting finite element schemes is presented indicating order optimal convergence rates under appropriate smoothness assumptions. 
For the solution of the nonlinear discretized problems, we consider a Newton method with line search for which we establish global linear convergence with convergence rates that are independent of the discretization parameters. We further prove local quadratic convergence in a mesh-dependent neighborhood of the solution which becomes effective when high accuracy of the nonlinear solver is demanded.
The assumptions required for our analysis cover inhomogeneous, nonlinear, and anisotropic materials, which may arise in typical applications, including the presence of permanent magnets. 
The theoretical results are illustrated by numerical tests for some typical benchmark problems.
\end{abstract}

\maketitle 

\begin{quote}
\end{quote}

\section{Introduction}
\label{sec:intro}

Problems in nonlinear magnetostatics arise in the modeling of high-power low-frequency applications like electric machines or power transformers~\cite{Hrabovcova2020,Salon1995}. Their efficient simulation is required for optimization and control of such devices; see e.g.~\cite{Gangl2015,Lowther2012}.  
In this paper, we consider systems of the form 
\begin{alignat}{4}
\curl h &= j_s \quad &\text{in } \Omega, \qquad && h &= \denergy(b), \label{eq:1}\\
\div b &= 0 \quad & \text{in } \Omega, \qquad && b \cdot n &= 0 \quad \text{on } \partial\Omega. \label{eq:2}
\end{alignat}
As usual, $h$ and $b$ denote the magnetic field and flux densities, and $j_s=\curl h_s$ is the driving current which is represented here by a source field $h_s$; see \cite{Biro1999,Meunier2008}. 
Following~\cite{silvester91}, we describe the material behavior via a magnetic energy density~$\energy(\cdot)$, which may additionally vary in space, so that $h=\denergy(b)$ is to be understood as a short-hand notation for $h(x)=\partial_b \energy(x,b(x))$. 
This form is general enough to account for inhomogeneous, nonlinear, anisotropic materials~\cite{silvester91}, and also allows to incorporate permanent magnets, as illustrated in Section~\ref{sec:numerics}.
In the case of isotropic materials, one simply has $\energy(b) = \widetilde \energy(|b|)$ and hence $h=\nu(|b|) \, b$ with chord reluctivity $\nu(s)=\frac{\widetilde \energy'(s)}{s}$; see \cite{heise94,Meunier2008} for instance.

\subsection*{Energy-based approach.}
Following standard computational practice, we use a vector potential $a$ in the sequel to represent the magnetic field in the form
\begin{align} \label{eq:3}
b = \curl a.
\end{align}
This allows to reduce the system \eqref{eq:1}--\eqref{eq:2} to the well-known vector potential formulation of magnetostatics; see e.g. \cite{Bossavit1998,Meunier2008} and \cite{AlonsoValli2010,Biro1999} for extensions to magneto-quasistatics.
Instead of treating the resulting nonlinear differential equations directly, we here use the fact that the solution of \eqref{eq:1}--\eqref{eq:2} can be characterized equivalently as minimizer of a magnetic energy functional, i.e. 
\begin{align} \label{eq:4}
\min_{a \in V_0} \int_\Omega \energy(\curl a) - h_s \cdot \curl a \, dx.
\end{align}
The additional boundary and gauging conditions required for unique solvability are incorporated into the set $V_0$ over which the minimization takes place; see Section~\ref{sec:analysis} for details.
This \emph{variational formulation} of nonlinear magnetostatics allows for a systematic analysis of the problem and further opens the way for the efficient numerical approximation by finite element methods and iterative solvers; see~\cite{Bossavit1998,Meunier2008} and the references given there.
Similar arguments also apply to the scalar potential formulation of nonlinear magnetostatics~\cite{Engertsberger23,Meunier2008}. A detailed analysis for the latter will be presented elsewhere. 
\subsection*{Scope and main contributions.}
We here consider the systematic numerical approximation of \eqref{eq:1}--\eqref{eq:2} by higher--order finite elements~\cite{Demkowicz2008,Zaglmayr2005}. Numerical quadrature is used to efficiently handle the nonlinear terms and domain mappings are employed to represent geometries with curved boundaries. A Newton method with line search is considered for the iterative solution of the resulting finite dimensional minimization problems. 
The variational form \eqref{eq:4} of our model problem will play a key role in the numerical analysis of the approximation as well as the iterative solution process.
Our first contributions are
\begin{itemize}
\item to prove well-posedness of the Galerkin approximation for \eqref{eq:4} under general assumptions on the problem data and the discretization;
\item to establish order optimal error estimates in the presence of quadrature errors and under suitable smoothness assumptions for the true solution.
\end{itemize} 
Finite element methods for Maxwell's equations are widely used in computational practice~\cite{Geuzaine2007,FEMM,Schoeberl2014} and they have also been studied intensively in the literature; see~\cite{Bossavit1998,Demkowicz2008,Monk2003} for an introduction and further references. 
An extensive amount of work exists concerning error estimates for linear problems; see e.g. \cite{Boffi2013,Hiptmair2002,Monk2003} and the references given there. 
Only a few results, however, seem available concerning a rigorous error analysis for nonlinear problems. 
Let us explicitly mention \cite{heise94}, where such an analysis is presented for problems in two space dimensions and for inhomogeneous but isotropic materials.
In this paper, we consider two- and three-dimensional problems, rather general material laws, and we also consider the effect of numerical quadrature and domain mappings, which are required to efficiently handle higher order approximations on curved domains. 
Our further contributions are
\begin{itemize}
\item to establish \emph{global linear convergence} of the damped Newton method with a convergence factor independent of the discretization parameters; 
\item to additionally prove \emph{local quadratic convergence} with convergence radius depending on the mesh size and polynomial degree.
\end{itemize}
While the first result is of key relevance for computations with lower--order finite elements, the latter property becomes particularly important if high accuracy in the nonlinear solvers is required.
The iterative solution of nonlinear systems arising in magnetostatics has of course also been studied intensively in the literature; see the references given below. 
We here consider more general material laws and prove global convergence of the proposed algorithms, whose design and analysis are strongly based on the variational formulation \eqref{eq:4} of the problem and the corresponding discretization schemes. 
We show that the convergence is at least linear with a convergence factor independent of the meshsize and the polynomial degree, while the onset of quadratic convergence depends on the discretization parameters.
As observed by Dular et al.~\cite{Dular2020}, the convergence behavior additionally depends on the nonlinearity of the material law.
A similar analysis can be performed for other iterative methods, including fixed-point iterations~\cite{Dlala2008}, the Kacanov iteration~\cite{Friedrich2019}, and different variants of the Newton method~\cite{Borghi2004,Fujiwara2005,Sande2003}. A related analysis for the scalar potential formulation of magnetostatics was given in~\cite{Engertsberger23} and for abstract variational problems in~\cite{Heid2023}. 

\subsection*{Outline of the manuscript.}
In Section~\ref{sec:analysis}, we introduce our main assumptions, and then formally state and briefly analyze the vector potential problem~\eqref{eq:1}--\eqref{eq:2} and its variational form~\eqref{eq:4}. 
Section~\ref{sec:discretization} introduces the finite element approximation of the problem and presents its error analysis. 
For the iterative solution of the discretized minimization problem, we consider in Section~\ref{sec:newton} a Newton method with line search, and we prove global linear convergence with a convergence factor independent of the discretization parameters. 
In Section~\ref{sec:quadratic}, we further establish the local quadratic convergence of the Newton method and derive bounds for the convergence radius. 
The extension of our results to domains with curved boundaries and to problems in two space dimensions is possible and will be briefly discussed in Section~\ref{sec:general}.
For an illustration of our theoretical findings, we report in Section~\ref{sec:numerics} about computational results for some typical test problems in two and three space dimensions.

\section{Preliminaries} \label{sec:analysis}
In the following four sections, we first present and then prove our main results for problems on polyhedral domains in three space dimensions. Domains with curved boundaries and in two-space dimensions will be discussed in Section~\ref{sec:general}.

\subsection{Notation}
Let $\Omega \subset \RR^3$ be some open set. We write $L^2(\Omega)$ for the space of square integrable functions and denote by $\langle f,g\rangle_\Omega = \int_\Omega f \cdot g \, dx$ the scalar product of $L^2(\Omega)$ and $L^2(\Omega)^3$, respectively. 
The associated norms are designated by $\|\cdot\|_{L^2(\Omega)}$.
The spaces $H^k(\Omega)$ consist of those functions with square integrable derivatives of order $\le k$. By $H(\curl;\Omega)$ and $H(\div;\Omega)$, we denote the space of functions in $L^2(\Omega)^3$ with $\curl u \in L^2(\Omega)^{3}$ and $\div u \in L^2(\Omega)$, respectively; see e.g.~\cite{Boffi2013,Monk2003}. 
We write $H_0(\curl;\Omega)$ and $H_0(\div;\Omega)$ for the subspaces of $H(\curl;\Omega)$ and $H(\div;\Omega)$ with vanishing tangential resp. normal traces at the boundary. 
To guarantee uniqueness and consistency of the representation \eqref{eq:3}, we require the domain $\Omega$ to be of trivial topology and consider the stable decomposition 
\begin{align} \label{eq:decomp}
H_0(\curl;\Omega) = \nabla H_0^1(\Omega) \oplus V_0.
\end{align} 
The space $V_0$ can be chosen, e.g., as the orthogonal complement of $\nabla H_0^1(\Omega)$, 
but all the following results will be independent of the particular choice of $V_0$.
Note that $\|v\|_{\curl} := \|\curl v\|_{L^2}$ defines a norm on $V_0$; see \cite{Arnold2019,Monk2003}. 

\subsection{Main assumptions and well-posedness}
In the main parts of our analy\-sis, we use the following conditions for the problem data.
\begin{assumption} \label{ass:1}
$\Omega \subset \RR^3$ is a bounded Lipschitz domain of trivial topology, i.e., simply connected with connected boundary.
The energy density $\energy : \overline \Omega \times \RR^3 \to \RR$ is piecewise continuous with respect to the first argument and satisfies 
\begin{itemize}
\item $\energy(x,\cdot) \in C^2(\RR^3)$,
\qquad  $|\energy(x,0)| + |\denergy(x,0)| \le C$, 
\smallskip
\item 
$\gamma\,|\xi|^2 \leq \langle \ddenergy(x,\eta) \, \xi, \xi \rangle \leq L \,|\xi|^2 
\quad \forall \xi,\eta \in \mathbb{R}^3$,
\end{itemize}
for all $x \in \overline \Omega$ with uniform constants $L,\gamma>0$. 
The source current density finally satisfies $j_s=\curl h_s$ for some $h_s \in H(\curl)$, and hence $\div j_s=0$.
\end{assumption}
\begin{remark} \label{rem:monotone}
For every point in space, the energy functional $\energy(x,\cdot)$ is smooth, strongly coercive, and quadratically bounded. Furthermore
\begin{align} \label{eq:monotone}
\gamma |a-b|^2 \le \langle \denergy(x,a) - \denergy(x,b), a-b \rangle \le L |a-b|^2 
\end{align}
holds uniformly for all $x \in \Omega$ and all $a,b \in \RR^3$, i.e., $\denergy(x,\cdot)$ is strongly monotone and Lipschitz continuous~\cite{Zeidler2A}. 
In our analysis, we will usually drop the spatial dependence and simply write $\energy(b)$ instead of $\energy(x,b)$ and so on. 
\end{remark}
The above assumptions allow us to establish the well-posedness of problem \eqref{eq:1}--\eqref{eq:2}. For completeness and later reference, we state the result in detail. 
\begin{theorem} \label{thm:primal}
Let Assumption~\ref{ass:1} hold. 
Then the nonlinear variational problem 
\begin{align} \label{eq:primal}
\min_{a \in V_0} \int_\Omega \energy(\curl a) - h_s \cdot \curl a \, dx
\end{align}
has a unique solution $a \in V_0$ which is characterized equivalently as the unique solution to the variational identity
\begin{align} \label{eq:var}
\langle \denergy(\curl a), \curl v \rangle_\Omega &= \langle h_s , \curl v \rangle_\Omega \qquad \forall v \in V_0.
\end{align}
The functions $b= \curl a$ and $h=\denergy(b)$ in turn correspond to the unique weak solution of the boundary value problem \eqref{eq:1}--\eqref{eq:2}. 
\end{theorem}
\begin{proof}
By our assumptions, the system~\eqref{eq:primal} amounts to a convex minimization problem over the Hilbert spaces $V_0 \subset H_0(\curl;\Omega)$, with \eqref{eq:var} denoting the necessary and sufficient optimality conditions.  
Existence of a unique solution to \eqref{eq:var}, on the other hand, can be established by the Zarantonello lemma; see \cite{Zeidler2A}.  
For convenience of the reader, we recall the main arguments: 
Let $\tau>0$ and
\begin{align}
\Phi_\linep : V_0 \to V_0, \quad a\mapsto w:=\Phi_\linep(a)
\end{align}
be defined implicitly via the variational problem 
\begin{align}
\langle \curl w,\curl v\rangle_\Omega 
= \langle\curl a ,\curl v\rangle_{\Omega} -\linep \langle \denergy(\curl a) - h_s, \curl v \rangle_\Omega
\end{align}
for all $v \in V_0$.
The Lax-Milgram lemma~\cite[Theorem 18.E]{Zeidler2A} provides existence of a unique solution $w\in V_0$ for any $a \in V_0$ given. Hence $\Phi_\linep : V_0 \to V_0$ is well--defined. For sufficiently small step size $\linep>0$, it can further be shown to be contractive. 
To see this, let $a_1$, $a_2 \in V_0$ be given and expand
\begin{align*}
&\|\curl \Phi_\linep(a_1) - \curl  \Phi_\linep(a_2)\|_{L^2(\Omega)}^2 \\
&= \|\curl a_1 - \curl a_2\|^2_{L^2(\Omega)} 
+ \linep^2 \|\denergy(\curl a_1) - \denergy(\curl a_2)\|_{L^2(\Omega)}^2\\ 
&\qquad \qquad \qquad - 2 \linep \langle \denergy(\curl a_1)-\denergy(\curl a_2), \curl a_1 - \curl a_2\rangle_\Omega.
\end{align*}
The conditions on the energy functional $\energy(\cdot)$ in Assumption~\ref{ass:1} guarantee global Lipschitz continuity $ \|\denergy(u)-\denergy(v)\|_{L^2(\Omega)} \le L \|u-v\|_{L^2(\Omega)}$ and uniform monotonicity $\langle \denergy(u) - \denergy(u), u-v \rangle_\Omega \ge \gamma \|u-v\|_{L^2(\Omega)}^2$
for all $u,v \in L^2(\Omega)^3$; see Remark~\ref{rem:monotone}. 
Together with the previous formula, this yields
\begin{align*}
\|\curl (\Phi_\linep(a_1) - \Phi_\linep(a_2))\|_{L^2(\Omega)}^2 \le (1-2 \linep \gamma + \linep^2 L^2 ) \|\curl ( a_1 - a_2)\|_{L^2(\Omega)}^2. 
\end{align*}
Since $\|\curl \cdot\|_{L^2(\Omega)}$ is a norm on $V_0$, we see that the mapping $\Phi_\linep$ is contractive on $V_0$ 
for any $0 < \linep < 2 \gamma/L^2$. 
Hence the fixed-point problem $a = \Phi_\linep(a)$ has a unique solution $a \in V_0$ for such $\tau$. 
From $b=\curl a$, we further obtain $\div b = 0$, and using $h=\denergy(b)= \denergy(\curl a)$, we see that \eqref{eq:var} amounts to the weak form of~\eqref{eq:2}. 
Hence any sufficiently regular solution of \eqref{eq:1}--\eqref{eq:2} also solves the variational principle~\eqref{eq:var}, and vice versa.  
\end{proof}

\begin{remark} \label{rem:all}
For later reference, let us also mention the following simple fact. 
Since $H_0(\curl;\Omega) = \nabla H_0^1(\Omega) \oplus V_0$ and $\curl (\nabla \phi)=0$ for all $\phi \in H_0^1(\Omega)$,  the variational identity \eqref{eq:var} actually holds for all $v \in H_0(\curl;\Omega)$.    
\end{remark}

\section{Finite element method}
\label{sec:discretization}

For the numerical solution of the magnetostatic problem \eqref{eq:1}--\eqref{eq:2}, we consider finite element approximations of the variational problems \eqref{eq:primal} resp. \eqref{eq:var}. Numerical quadrature is used for handling the nonlinear terms. 
For ease of presentation, we assume in this section that the domain $\Omega$ is polyhedral. 
The extension to domains with curved boundaries will be discussed in Section~\ref{sec:curved}.

\subsection{Preliminaries}
Let $\Th$ be a tetrahedral finite mesh of $\Omega$ and $P_k(\Th)$ denote the space of piecewise polynomials of degree $\le k$ on  $\Th$. 
Further, let
\begin{align}
\langle u,v\rangle_h = \sum\nolimits_{T \in \T_h} \sum\nolimits_{j=1}^\ell u(x_{T,j}) \cdot v(x_{T,j}) \, \hat w_{j} |T| 
\end{align}
denote an approximation for the $L^2$-scalar product $\langle u,v\rangle_\Omega$ obtained by applying some quadrature rule to the integration on every element $T$.  
We write $\|\cdot\|_h$ for the corresponding norm. 
Without further mentioning, we assume that the integration points $x_{T,j}=\phi_T(\hat x_j)$ are mapped from the reference element $\hat T$ in the usual manner; see \cite{Ciarlet2002} for details.
Throughout our analysis, we make the following additional assumptions concerning the discretization. 
\begin{assumption} \label{ass:2}
$\T_h$ is a geometrically conforming and uniformly shape-regular simplicial partition of $\Omega$; see ~\cite{Ciarlet2002}. 
The local quadrature rule $(\hat x_{j},\hat w_j)$ has positive weights and is exact for piecewise polynomials of degree $\le 2k$ on $T$, i.e., 
\begin{align*}
\langle u,v\rangle_{h} = \langle u,v\rangle_\Omega \qquad \forall u,v \in P_k(\T_h)^3. 
\end{align*}
Further let $W_h = \mathcal{N}_k(\mathcal{T}_h) \cap H_0(\curl;\Omega)$ be the Nedelec space of order $k$; see \cite{Boffi2013,Nedelec1980}. 
This space is decomposed as $W_h = \nabla S_h \oplus V_h$
with $S_h = P_{k+1}(\mathcal{T}_h) \cap H_0^1(\Omega)$. 
We assume that this splitting is stable, i.e., $\|v_h\|_{H(\curl)} \le C_* \|\curl v_h\|_{L^2}$ for some $C_*>0$ and for all $v_h \in V_h$. 
Furthermore, assume that $h_s \in H(\curl;\Omega) \cap H^2(\mathcal{T}_h)^3$.
\end{assumption}
\begin{remark}
The latter condition allows us to evaluate the source field $h_s(x)$ for all $x \in \Omega$, which is required to perform numerical quadrature.  
An efficient implementation of a basis for the reduced Nedelec space $V_h$ can be obtained by tree-cotree gauging and hierarchical basis constructions; see \cite{Albanese1988,Dlotko2018} and \cite{Zaglmayr2005,Rapetti2022} for details.
Further note that 
\begin{align} \label{eq:Zh}
\curl(V_h) = \{z_h \in Z_h: \div z_h=0\},
\end{align}
where $Z_h=\RT_k(\Th) \cap H_0(\div;\Omega)$ is the Raviart-Thomas space of order $k$; see again \cite{Boffi2013} for a definition of this space.
From the exactness properties of the quadrature rule, one can infer that  $\|\curl v_h\|_h = \|\curl v_h\|_{L^2}$ for all $v_h \in V_h$, and hence  $\|\curl v_h\|_h$ defines a norm on the approximation space $V_h$.
\end{remark}

\subsection{Finite element method}
For the actual discretization of our model problem \eqref{eq:1}--\eqref{eq:2}, we consider the finite dimensional minimization problem
\begin{align} \label{eq:primalh}
\min_{a_h \in V_h} \langle w(\curl a_h), 1 \rangle_h - \langle h_s,\curl a_h\rangle_h,
\end{align}
which can be interpreted as an inexact Galerkin approximation of \eqref{eq:primal}. 
We note that this problem is suitable for an efficient implementation.
Let us start with establishing the well-posedness of the discretized problem. 
\begin{theorem} \label{thm:primalh}
Let Assumptions~\ref{ass:1} and \ref{ass:2} be valid. Then the finite--dimensional minimization problem \eqref{eq:primalh} has a unique solution $a_h \in V_h$. 
\end{theorem}
\begin{proof}
By our assumptions, problem \eqref{eq:primalh} is convex. The corresponding necessary and sufficient first-order optimality conditions for a minimum read 
\begin{align} \label{eq:varh}
\langle \denergy(\curl a_h), \curl v_h\rangle_h = \langle h_s, \curl v_h\rangle_h \qquad \forall v_h \in V_h.
\end{align}
The existence of a unique solution $a_h \in V_h$ to this finite--dimensional nonlinear problem follows with the same arguments as used in the proof of Theorem~\ref{thm:primal} and using that $\|\curl v_h\|_h$ defines a norm on $V_h$; see Assumption~\ref{ass:2}. 
\end{proof}

\subsection{Error estimates}
We start by collecting some auxiliary results about projection operators that will be required in our analysis below. 
\begin{lemma} \label{lem:proj}
Let Assumption~\ref{ass:2} hold and let $V_0$, $V_h$, and $Z_h$ be the function spaces introduced above. 
Then there exist projection operators $\pi_h^* : V_0 \cap H^2(\Th)^3 \to V_h$ and $\tilde \pi_h^* : H_0(\div) \cap H^1(\Th)^3 \to Z_h$ such that $\curl(\pi_h^* a) = \tilde \pi_h^* \curl(a)$. Moreover
\begin{align}
\|\curl (\pi_h^* a - a)\|_{L^2} + \|\curl (\pi_h^* a - a)\|_h &\le C  h^j \|\curl a\|_{H^j(\Th)} \\
\|\tilde \pi_h^* b - b\|_{L^2} + \|\tilde \pi_h^* b - b\|_h &\le C  h^j \|b\|_{H^j(\Th)}
\end{align}
for all $2 \le j \le k+1$ and all $a \in V_0 \cap H^j(\Th)^3$, $b \in H_0(\div) \cap H^j(\Th)^3$.
\end{lemma}
\begin{proof}
We choose $\tilde \pi_h^*$ as the standard Raviart-Thomas projection operator, which already implies the second error estimate; see \cite{Boffi2013} for details. 
The projection operator $\pi_h^*$ can then be constructed as follows: Let $\pi_h : H_0(\curl) \cap H^2(\Th)^3 \to W_h$ denote the standard Nedelec projection operator and $a \in V_0 \cap H^2(\Th)$. 
Then the discrete function $\pi_h a \in W_h$ can be split uniquely as $\pi_h a = a_h^* + \nabla \phi_h$ with $\phi_h \in S_h = H_0^1(\Omega) \cap P_k(\Th)$ and $a_h^* \in V_h$.
We define $\pi_h^* a := a_h^*$ and obtain
\begin{align*}
\curl (\pi_h^* a) 
= \curl (a_h^*)
= \curl (a_h^* + \nabla \phi_h)
= \curl (\pi_h a)
= \tilde \pi_h^* \curl(a).
\end{align*}
The last identity results from the commuting diagram property of the Nedelec and Raviart-Thomas projection operators; see~\cite[Ch.~2]{Boffi2013}. 
The error estimate for the projection $\pi_h^*$ claimed in the lemma finally follows from that for $\tilde \pi_h^*$.
\end{proof}
We can now already establish our first main result, which is concerned with quantitative error estimates for the proposed discretization scheme. 
\begin{theorem} \label{thm:estimates}
Let Assumptions~\ref{ass:1} and \ref{ass:2} hold and $a,a_h$ denote the unique solutions of \eqref{eq:primal} and \eqref{eq:primalh}. 
Further set $b=\curl a$, $h=\denergy(b)$ and $b_h=\curl a_h$, $h_h=\denergy(b_h)$, and assume that $b,h,h_s \in H^j(\Th)$ for some $2 \le j \le k+1$. 
Then
\begin{align*}
\|b - b_h\|_{L^2(\Omega)} + \|h - h_h\|_{L^2(\Omega)}
\le C\,h^{j}\,(|b|_{H^j(\mathcal{T}_h)} + |h|_{H^j(\mathcal{T}_h)}),
\end{align*}
and the constant $C$ is independent of the mesh size and the particular solution. 
\label{theorem:discretization:error}
\end{theorem}
\begin{proof}
We start with the estimate for $b-b_h=\curl(a-a_h)$, which will be the main step in the proof. 
By the triangle inequality, we can decompose
\begin{align}
\|b-b_h\|_{L^2(\Omega)} 
&\le \|b-b_h^*\|_{L^2(\Omega)} + \|b_h - b_h^*\|_{L^2(\Omega)} = (i) + (ii),     
\label{eqn:discrete:est:1}
\end{align}
with $b_h^* = \tilde \pi_h^* b = \tilde \pi_h^* \curl a = \curl \pi_h^* a$ denoting the projection as defined in the previous lemma. 
The first term can then be bounded by $(i) \le C h^j \|b\|_{H^j(\Th)}$
using Lemma~\ref{lem:proj}. 
For the second term, we get 
\begin{align*}
\gamma \, (ii)^2 
&=\gamma \| b_h - b_h^* \|_{L^2(\Omega)}^2 
= \gamma \|b_h - b_h^*\|_h^2 
\le \langle \denergy(b_h) - \denergy(b_h^*), b_h - b_h^* \rangle_h,
\end{align*}
where we used the exactness of the quadrature rule in the second and the monotonicity of the energy functional in the third step. 
From the variational identities \eqref{eq:var} and \eqref{eq:varh}, using $b_h=\curl a_h$, $b_h^* = \curl a_h^*$, and the exactness of the quadrature rule once again, we may further deduce that 
\begin{align*}
\gamma \, (ii)^2 
&\leq \langle h_s, \curl (a_h - a_h^*) \rangle_h - \langle h_s, \curl (a_h - a_h^*) \rangle_\Omega\\
& \qquad + \langle \denergy(\curl a) - \denergy(\curl a_h^*), \curl (a_h - a_h^*)\rangle_h \\
& \qquad + \langle \denergy(\curl a), \curl(a_h - a_h^*)\rangle_\Omega -  \langle \denergy(\curl a), \curl(a_h - a_h^*)\rangle_h\\
&= (iii) + (iv) + (v).
\end{align*}
Now let $\hat \pi_h : L^2(\Omega)^3 \to P_k(\Th)^3$ denote the elementwise $L^2$-projection. Then by the exactness of the quadrature rule, reverting to $b_h$, $b_h^*$, and using the Cauchy-Schwarz inequality and standard projection error estimates~\cite{Boffi2013,Brenner2008}, we get
\begin{align*}
(iii) = \langle h_s - \hat \pi_h h_s, \curl (a_h - a_h^*)\rangle_h 
&\le C h^j \|h_s\|_{H^j(\Th)} \|b_h - b_h^*\|_h.
\end{align*}
By the Lipschitz continuity of $\denergy(\cdot)$, we further see that
\begin{align*}
(iv) &\le L \|\curl a - \curl a_h^*\|_h \|\curl (a_h - a_h^*)\|_h 
\le C h^j \|b\|_{H^j(\Th)} \|b_h - b_h^*\|_h.
\end{align*}
For the second step, we here used $b=\curl a$, $\curl a_h^* = \tilde \pi_h^* \curl a = \tilde \pi_h^* b$ and the estimates of Lemma~\ref{lem:proj}. 
For the last term, we again employ the $L^2$-projection operator $\hat \pi_h$ and the exactness of the quadrature rule, to see that
\begin{align*}
(v) &= \langle \hat \pi_h \denergy(\curl a)-\denergy(\curl a), \curl(a_h - a_h^*)\rangle_h \\
&\le \|\hat \pi_h \denergy(b)-\denergy(b)\|_h \|b_h - b_h^*\|_h 
\le C h^j \|\denergy(b)\|_{H^j(\Th)} \|b_h - b_h^*\|_h.
\end{align*}
By assumption, $h=\denergy(b)$ has the required regularity. 
The exactness of the quadrature rule and the definition of $b_h=\curl a_h$ and $b_h^* = \curl a_h^*$ then further yield $\|\curl (a_h - a_h^*)\|_h = \|b_h - b_h^*\|_{L^2(\Omega)}$. 
In summary, we thus see that 
\begin{align*}
\gamma \|b_h - b_h^*\|_{L^2(\Omega)}^2 
\le Ch^j (\|h_s\|_{H^j(\Th)} + \|b\|_{H^j(\Th)} + \|h\|_{H^j(\Th)}) \|b_h - b_h^*\|_{L^2(\Omega)},
\end{align*}
which implies $(ii) \le C'h^j (\|h_s\|_{H^j(\Th)} + \|b\|_{H^j(\Th)} + \|h\|_{H^j(\Th)})$. Together with the estimate for the term $(i)$, this already yields the error bound for magnetic flux. 
By the Lipschitz continuity of $\denergy(\cdot)$, we finally get 
\begin{align*}
\|h-h_h\|_{L^2(\Omega)} 
&= \|\denergy(b) - \denergy(b_h)\|_{L^2(\Omega)} 
\le L \|b-b_h\|_{L^2(\Omega)},
\end{align*}
which yields the corresponding estimate for the error in the magnetic field.
\end{proof}

\begin{remark}\label{remark:convergence:b:h}
The assumptions on $b$ and $h$ implicitly encode certain regularity conditions on $\denergy(\cdot)$ which, however, do not appear explicitly in our analysis.  
This greatly simplifies our arguments compared to, e.g. \cite{heise94}.
Further note that, in view of approximation properties, the regularity conditions for the problem data are also necessary to obtain the predicted convergence rates. 
\end{remark}

\section{Global convergence of Newton's method}
\label{sec:newton}

We now study the iterative solution of the discretized nonlinear variational problem~\eqref{eq:primalh}. 
Since this problem is smooth and convex, we employ the Newton method with line search~\cite{Deuflhard2011,Nocedal2006}. 
The iteration thus takes the form 
\begin{align} \label{eq:opt1}
a_h^{n+1} = a_h^n + \linep^n \delta a_h^n, \qquad n \ge 0,
\end{align}
where $a_h^0 \in V_h$ is given and the increment $\delta a_h^n \in V_h$ is defined as the solution of  
\begin{align} \label{eq:opt2}
\langle \ddenergy(b_h^n) \curl \delta a_h^n, \curl v_h\rangle_h = - \langle \denergy(b_h^n) - h_s, \curl v_h\rangle_h 
\end{align}
for all $v_h \in V_h$ with $b_h^n = \curl a_h^n$ introduced for abbreviation.
The step size $\linep^n$ will be chosen by Armijo back-tracking~\cite[Ch.~3]{Nocedal2006}, i.e., by the rule
\begin{align} \label{eq:opt3}
\linep^n &= \max\{\linep=\rho^k: k \ge 0 \quad \text{such that} \\ 
&\qquad \qquad W(a_h^n+\linep \delta a_h^n) \le W(a_h^n)  + \sigma \linep \langle \denergy(b_h^n) - h_s, \curl \delta a_h^n\rangle_h\} \notag
\end{align}
with parameters $0<\rho \leq 1/2$ and $0 < \sigma < 1/2$. 
Here and below, we denote by
\begin{align} \label{eq:energy}
W(v_h) := \langle \energy(\curl v_h),1\rangle_h - \langle h_s,\curl v_h\rangle_h,
\end{align}
the discrete magnetic energy, which will play an important role in the subsequent analysis.
The goal of this section is to prove the following result. 
\begin{theorem}[Global linear convergence] \label{thm:global} $ $ \\
Let Assumptions~\ref{ass:1} and \ref{ass:2} hold, and $a_h$ denote the unique solution of the discrete variational problem \eqref{eq:primalh}. 
Furthermore, let $a_h^n$, $n \ge 0$ be the sequence of iterates generated by \eqref{eq:opt1}--\eqref{eq:opt3} with initial value $a_h^0 \in V_h$. 
Then, if $\delta a_h^n=0$ for some $n=n^*<\infty$, one has $a_h^n = a_h$; otherwise $n^*=\infty$.
Moreover
\begin{align}
\|\curl a_h^n - \curl a_h\|_{L^2(\Omega)}^2 \le C\,q^n \|\curl a_h^0 - \curl a_h\|_{L^2(\Omega)}^2 
\end{align}
for all $n < n^*+1$ with $C=\frac{L}{\gamma}$ and contraction factor $q=1- 4 \rho \sigma (1-\sigma) \frac{\gamma^3}{L^3} < 1$.
\end{theorem}
\begin{remark}
The theorem shows that the iterates $a_h^n$ converge at least \emph{$r$-linearly} to the discrete solution $a_h$. 
Let us emphasize that the convergence is global and the convergence factor $q$ is independent of the discretization parameters. 
In the next section, we will also establish local quadratic convergence, but with a convergence radius depending on the mesh size.
\label{remark:meshindependent}
\end{remark}

%
The remainder of this section is devoted to the proof of the previous theorem. For the convenience of the reader, we split it into several steps. 
\subsection*{Step~1.}
We first ensure that the increment $\delta a_h^n$ is well-defined. To see this, let us note that $\ddenergy(x,\eta)$ is symmetric, uniformly bounded, and positive definite. Hence \eqref{eq:opt2} amounts to a linear elliptic variational problem. Existence of a unique solution $\delta a_h^n \in V_h$ can thus again be established by the Lax--Milgram theorem~\cite{Zeidler2A}.
Moreover, we see that 
\begin{align}
\langle \denergy(b_h^n) - h_s, \curl \delta a_h^n\rangle_h
&= -\langle \ddenergy(b_h^n) \curl \delta a_h^n, \curl \delta a_h^n\rangle_h \\
&\le -\gamma \|\curl \delta a_h^n\|^2_h, \notag
\end{align}
and hence $\delta a_h^n$ is a descent direction~\cite{Nocedal2006} for the minimization problem \eqref{eq:primalh}.
As a second step, we show that Armijo back-tracking yields a reasonable step size. 
\begin{lemma} \label{lem:omega}
Let Assumptions~\ref{ass:1} and \ref{ass:2} hold, and $0 \ne \delta a_h^n \in V_h$. Then the parameter rule \eqref{eq:opt3} defines a unique step size $\linep^n$ satisfying 
\begin{align}
0 < \linep_* \le \linep^n \le 1
\end{align}
with lower bound $\linep_* = 2 \rho (1-\sigma) \gamma/L$ independent of the discretization space $V_h$.
\end{lemma}
\begin{proof}
We abbreviate $\lines(\linep)=W(a_h^n+\linep \delta a_h^n)$. 
By Taylor expansion, we get
\begin{align*}
\lines(\linep) 
= \lines(0) + \sigma \linep \lines'(0) + \delta(\linep)
\end{align*} 
with remainder term $\delta(\linep)=(1-\sigma) \linep \lines'(0) + \tfrac{\linep^2}{2} \lines''(\xi_\linep)$ and $\xi_\linep \in (0,\linep)$. 
Using elementary computations, one can verify that
\begin{align*}
\lines'(0) &= -\langle \ddenergy(b_h^n) \curl \delta a_h^n, \curl \delta a_h^n\rangle_h \qquad \text{and} \\
\lines''(\xi_\linep) &= \langle \ddenergy(b_h^n(\xi_\linep) \curl \delta a_h^n, \curl \delta a_h^n\rangle_h 
\end{align*}
with $b_h^n = \curl a_h^n$ and $b_h^n(\xi_\linep) = \curl (a_h^n + \xi_\linep \delta a_h^n)$ used for abbreviation. 
By Assumption~\ref{ass:1}, we may thus conclude that 
\begin{align*}
\delta(\linep) 
\le \left(-\gamma (1-\sigma) \linep + L \tfrac{\linep^2}{2} \right) \, \|\curl \delta a_h^n\|_h^2. 
\end{align*}
For $\linep < 2 (1-\sigma) \gamma/L$, the remainder $\delta(\linep)$ becomes negative, i.e., the step size~$\linep^n$ chosen by \eqref{eq:opt3} is certainly larger then $\rho$ times this bound. 
\end{proof}

\subsection*{Step~2.}
The next ingredient for our analysis is an equivalence between the squared norm distance to the solution and the difference in the energy. 
\begin{lemma} \label{lem:equi}
Let Assumptions~\ref{ass:1} and \ref{ass:2} hold, and let $a_h \in V_h$ denote the unique solution of \eqref{eq:primalh}.
Then for all $v_h \in V_h$ there holds
\begin{align} \label{eq:equi}
\frac{\gamma}{2} \|\curl (v_h - a_h)\|_h^2 \le W(v_h) - W(a_h) \le \frac{L}{2} \|\curl (v_h - a_h)\|_h^2.
\end{align}
\end{lemma}
\begin{proof}
Let $\lines(t) = W(a_h + t (v_h-a_h))$. Then by Taylor expansion, one can verify that
$W(v_h) - W(a_h) = \lines(1)-\lines(0) = \lines'(0) + \frac{1}{2} \lines''(\xi)$ for some $\xi \in (0,1)$. 
From the definition of $\lines(\cdot)$ and $W(\cdot)$, we see that 
\begin{align*}
\lines'(0)=\langle \denergy(\curl a_h) - h_s,\curl (v_h-a_h)\rangle_h=0,
\end{align*}
since $a_h$ is the solution of the discrete minimization problem \eqref{eq:primalh}. 
%
Moreover 
\begin{align*}
\lines''(\xi) = \langle \ddenergy(b_h(\xi)) \curl (v_h - a_h), \curl (v_h - a_h)\rangle_h,
\end{align*}
with $b_h(\xi) = \curl a_h + \xi  \curl (v_h - a_h)$. 
The estimates of the lemma then follow immediately from the bounds for $\ddenergy(\cdot)$ provided by Assumption~\ref{ass:1}.
\end{proof}

\subsection*{Step~3.}
We can now establish convergence of the energy $W(a_h^n)$ obtained by the damped Newton iteration \eqref{eq:opt1}--\eqref{eq:opt3} towards the minimal value $W(a_h)$.
\begin{lemma} \label{lem:decay}
Let Assumption~\ref{ass:1} and \ref{ass:2} be valid, and $a_h$ denote the unique solution of \eqref{eq:primalh}. 
Further let $a_h^n$, $n \ge 0$ be the iterates obtained by \eqref{eq:opt1}--\eqref{eq:opt3} with $a_h^0 \in V_h$ given, and assume that $\delta a_h^n \ne 0$ for all $n < n^*$. Then 
\begin{align*}
W(a_h^n) - W(a_h) \le q^n \, [W(a_h^0) - W(a_h)] \qquad \text{for all } n < n^*+1
\end{align*}
with a uniform contraction factor $q := 1- 4 \rho \sigma (1-\sigma) \gamma^3/L^3 < 1$.
\end{lemma}
\begin{proof}
Let us abbreviate $\nu_h^n := \ddenergy(b_h^n) = \ddenergy(\curl a_h^n)$, which by Assumption~\ref{ass:1} is uniformly positive definite and bounded. 
We may thus introduce the scaled norm $\|\curl v_h\|_{\nu_h^n}^2 := \langle \nu_h^n \curl v_h, \curl v_h\rangle_h$, which by Assumption~\ref{ass:1} satisfies
\begin{align} \label{eq:equiv}
\gamma \|\curl v_h\|^2_h \le \|\curl v_h\|_{\nu_h^n}^2 \le L \|\curl v_h\|^2_h.
\end{align}
Now let $0 \le n < n^*$. Then by assumption and the previous results, we know that $\delta a_h^n \ne 0$, and hence \eqref{eq:opt3} selects a step size $0 < \linep_* \le \linep^n \le 1$ such that
\begin{align}
W(a_h^{n+1}) - W(a_h) 
&\le W(a_h^n) - W(a_h) + \sigma \linep^n \langle \denergy(b_h^n) - h_s, \curl \delta a_h^n\rangle_h \notag \\
&= W(a_h^n) - W(a_h) - \sigma \linep^n \|\curl \delta a_h^n\|_{\nu_h^n}^2. \label{eq:inter}
\end{align}
In the last step, we used the definition \eqref{eq:opt2} of the Newton step and the special construction of the norm $\|\cdot \|_{\nu_h^n}$. 
From the definition of $\nu_h^n$, we further obtain 
\begin{align*}
\|\curl \delta a_h^n\|_{\nu_h^n} 
&= \sup_{v_h \in V_h} \frac{\langle \nu_h^n \curl \delta a_h^n, \curl v_h\rangle_h}{\|\curl v_h\|_{\nu_h^n}} \\
&= \sup_{v_h \in V_h} \frac{\langle \denergy(\curl a_h) - \denergy(\curl a_h^n), \curl v_h\rangle_h}{\|\curl v_h\|_{\nu_h^n}} \\
& \ge \frac{\gamma \|\curl (a_h - a_h^n)\|^2_h}{\|\curl (a_h - a_h^n)\|_{\nu_h^n}}\ge \frac{\gamma}{\sqrt{L}} \|\curl (a_h - a_h^n)\|_h.
\end{align*}
In the second step, we used the definition of the increment $\delta \psi_h^n$ and the definitions $b_h = \curl a_h$, resp. $b_h^n = \curl a_h^n$; in the third step, we used $v_h = a_h - a_h^n$ as a test function and the monotonicity of $\denergy(\cdot)$; in the last step, we employed the norm equivalence \eqref{eq:equiv}.
Together with the estimate of the previous lemma, we thus obtain $\|\curl \delta a_h^n\|_{\nu_h^n}^2 \ge \frac{2\gamma^2}{L^2} \big( W(a_h^n) - W(a_h)\big)$.
Inserting this into \eqref{eq:inter} and invoking the lower bound $\linep^n \ge \linep_*$ then yields 
\begin{align*}
W(a_h^{n+1}) - W(a_h) 
\le \left(1-\tfrac{2\sigma \linep_* \gamma^2}{L^2}\right) \, \big( W(a_h^n) - W(a_h)\big).   
\end{align*}
The assertion of the lemma now follows by inserting the definition of  $\linep_*$.
\end{proof}

\subsection*{Step~4.}
We can finally conclude the proof of Theorem~\ref{thm:global}. 
By combining the estimates of Lemma~\ref{lem:equi}~and~\ref{lem:decay}, and using the exactness of the quadrature rule, we immediately obtain the convergence estimates of the theorem. 
Now assume that $\delta a_h^n=0$ for some $n=n^* < \infty$. Then by \eqref{eq:opt2}  we can see that 
\begin{align*}
\langle \denergy(\curl a_h^n) - h_s, \curl v_h \rangle_h = -\langle \ddenergy(b_h^n) \curl \delta a_h^n, \curl v_h \rangle_h=0
\end{align*}
for all $v_h \in V_h$. As shown in the proof Theorem~\ref{thm:primalh}, this identity characterizes the unique minimizer of \eqref{eq:primalh}, and hence $a_h^n = a_h$.
\qed

\section{Local quadratic convergence}
\label{sec:quadratic}
For completeness of our analysis, we further establish local quadratic convergence of the Newton method \eqref{eq:opt1}--\eqref{eq:opt3}. 
In contrast to the global linear convergence, the onset of quadratic convergence will depend on the discretization parameters via an inverse inequality. 
As a preliminary step, we show that a step size $\linep^n=1$ is chosen by \eqref{eq:opt2} provided that $a_h^n$ is sufficiently close to $a_h$. 
\begin{lemma} \label{lem:tau1}
Let Assumptions~\ref{ass:1} and \ref{ass:2} hold, and 
let $\ddenergy(\cdot)$ be uniformly Lipschitz continuous with respect to 
the second argument with Lipschitz constant~$L''$. 
Furthermore, let $\|\curl (a_h^n - a_h^{n+1})\|_{L^\infty(\Omega)} \le \frac{(1-2\,\sigma)\,\gamma}{L''\,\linep_*}$, which is guaranteed to occur by Theorem~\ref{thm:global} for all $n \gg 1$. 
Then the step size rule \eqref{eq:opt2} returns $\linep^n = 1$. 
\end{lemma}
\begin{proof}
We proceed with similar reasoning as in the proof of Lemma~\ref{lem:omega} and abbreviate $\lines(\linep) = W(a_h^n + \linep \delta a_h^n)$. 
Then 
\begin{align*}
\lines(1) - \lines(0) - \sigma \lines'(0) =  (1-\sigma) \lines'(0) + \frac{1}{2} \lines''(\xi) =: \delta(1); 
\end{align*}
compare with the proof of Lemma~\ref{lem:omega}. 
By the definition of $\lines(\cdot)$ and using the definition \eqref{eq:opt2} of the Newton direction $\delta a_h^n$, one can rewrite $\delta(1)$ as
\begin{align*}
   \delta(1) = (\sigma - 1) \langle \ddenergy(b_h^n) \curl \delta a_h^n, \curl \delta a_h^n \rangle_h + \frac{1}{2} \langle \ddenergy(b_h^n(\xi)) \curl \delta a_h^n, \curl \delta a_h^n \rangle_h
\end{align*}
for $b_h^n = \curl a_h^n$ and $b_h^n(\xi) = \curl(a_h^n + \xi\,\delta a_h^n)$. By splitting the first term, we get
\begin{align*}
   \delta(1) &\leq \Big(\sigma - \frac{1}{2}\Big) \|\curl \delta a_h^n \|_{\nu_h^n}^2 + \frac{L''}{2} \xi \|\curl \delta a_h^n\|_{L^\infty(\Omega)} \langle \curl  \delta a_h^n ,\curl \delta a_h^n \rangle_h \\
   &\leq \Big(\sigma  + \frac{L''}{2 \gamma} \|\curl \delta a_h^n\|_{L^\infty(\Omega)} - \frac{1}{2}\Big) \|\curl \delta a_h^n \|_{\nu_h^n}^2.
\end{align*}
Hence we can see that $\delta(1)$ certainly becomes negative, whenever
\begin{align*}
    \|\curl \delta a_h^n\|_{L^\infty(\Omega)} \leq \frac{(1-2\,\sigma)\,\gamma}{L''}.
\end{align*}
Application of the estimate 
$\linep_* \|\curl \delta a_h^n\|_{L^\infty(\Omega)} \leq \|\curl (a_h^n -  a_h^{n+1})\|_{L^\infty(\Omega)}$ then 
already yields the stated result.
\end{proof}

\begin{remark}
From Theorem~\ref{thm:global}, we know that $\|\curl (a_h^n - a_h^{n+1})\|_{L^2\Omega)} \to 0$ with $n \to \infty$.
By using an inverse inequality \cite{Brenner2008,Ciarlet2002}, we may further deduce that 
\begin{align*}
\|\curl (a_h^n -a_h^{n+1})\|_{L^\infty(\Omega)} \le C(k) h_{\min}^{-3/2} \|\curl (a_h^n - a_h^{n+1})\|_{L^2(\Omega)} \to 0.
\end{align*}
Hence the conditions of the previous lemma are certainly reached after a finite number of iterations which, however, depend on the minimal mesh size $h_{\min}$ and the polynomial degree $k$ of the approximation. 
\end{remark}
We can now establish the local quadratic convergence of the Newton method. 
\begin{theorem} 
Let Assumption \ref{ass:1} and \ref{ass:2} be valid, and let $a_h$ denote the unique solution of problem \eqref{eq:primalh}. 
Further, assume that $\ddenergy(\cdot)$ is uniformly Lipschitz continuous with Lipschitz constant $L''$. 
Then 
\begin{align}
\|\curl (a_h^{n+1} - a_h)\|_{L^2(\Omega)} \leq M_*(h_{\min},k) \,\|\curl (a_h^n -  a_h)\|_{L^2(\Omega)}^2
\end{align}
holds with $M_*(h_{\min},k) = C(k) \frac{L''}{\gamma} h_{\min}^{-3/2}$, whenever $\|\curl (a_h^n - a_h) \|_{L^\infty(\Omega)}$ is sufficiently small.
In particular, the convergence of \eqref{eq:opt1}--\eqref{eq:opt3} is locally quadratic.
\end{theorem}
\begin{proof}
By the strong convexity of $\energy(\cdot)$ and the exactness of the quadrature rule required in our assumptions, we deduce that 
\begin{align*}
\gamma \|\curl (a_h^{n+1} - a_h)\|_{L^2(\Omega)}^2
&\le \langle \ddenergy(\curl a_h^n) \curl(a_h^{n+1}-a_h), \curl w_h\rangle_h = (*),
\end{align*}
where we introduced $w_h=a_h^{n+1}-a_h$ for abbreviation. 
Using that $\tau^n=1$, see Lemma~\ref{lem:tau1}, we know that $a_h^{n+1} = a_h^n + \delta a_h^n$, which yields
\begin{align*}
(*) &= \langle \ddenergy(\curl a_h^n) \curl(a_h^{n}-a_h), \curl w_h\rangle_h
+ \langle \ddenergy(\curl a_h^n) \curl \delta a_h^{n}, \curl w_h\rangle_h \\
&= \langle \ddenergy(\curl a_h^n) \curl(a_h^{n}-a_h) - \denergy(\curl a_h^n) + \denergy(\curl a_h), \curl w_h\rangle_h.
\end{align*}
For the last step, we used the definition of $\delta a_h^n$ and \eqref{eq:varh}. 
Note that for every smooth function $f : \RR^3 \to \RR^3$ with Lipschitz continuous derivative, we have 
\begin{align*}
|\partial_b f(b_1) \cdot (b_1-b_2) - f(b_1) + f(b_2)| \le \frac{L''}{2} |b_1-b_2|^2, 
\end{align*}
which follows from Taylor estimates. Applying this estimate to $f(b) = \denergy(x,b)$ with $b_1=\curl a_h^n(x)$ and $b_2=\curl a_h(x)$ for every $x \in \Omega$, we get
\begin{align*}
(*) \le \frac{L''}{2} \|\curl (a_h^n - a_h)\|_{L^\infty(\Omega)} \|\curl (a_h^n - a_h)\|_h \|\curl w_h\|_h. 
\end{align*}
Using exactness of the quadrature rule to replace $\|\cdot\|_h$ by $\|\cdot\|_{L^2(\Omega)}$, an inverse inequality~\cite{Ciarlet2002} to estimate $\|\cdot\|_{L^\infty(\Omega)}$ by $\|\cdot\|_{L^2(\Omega)}$, and canceling one term $\|\curl w_h\|_{L^2(\Omega)}$ on both sides, we finally obtain 
\begin{align*}
\|\curl (a_h^{n+1} - a_h)\|_{L^2(\Omega)} 
&\le \frac{L'' C(k)}{2\gamma} h_{min}^{-3/2} \|\curl (a_h^n - a_h)\|^2_{L^2(\Omega)},
\end{align*}
from which we conclude the local quadratic convergence.
\end{proof}

\begin{remark}
Let us emphasize that the linear convergence guaranteed by Theorem~\ref{thm:global} is global and with a convergence factor independent of the mesh size. The onset of the local quadratic convergence, on the other hand, is mesh dependent and of relevance only, if high accuracy of the iterative solver is required, which is the case, if a higher order approximations are employed; see Section~\ref{sec:numerics}.
\end{remark}

\section{Generalizations}
\label{sec:general}

In the following, we show how to generalize our main results to domains with curved boundaries as well as to problems in two space dimensions, which are of interest, e.g. in electric machine simulation~\cite{Hrabovcova2020,Salon1995}. 

\subsection{Curved domains}
\label{sec:curved}

Let Assumption~\ref{ass:2} be satisfied and thus $\Omega \subset \RR^3$ be polyhedral and $\Th$ a corresponding simplicial mesh. 
Further let $\phi : \Omega \to \Omega'$ be a diffeomorphism and piecewise smooth with repect to $\Th$. 
We call $\Omega$ the reference domain and $\Omega' = \phi(\Omega)$ the physical domain. In this physical domain, we consider the magnetostatic problem 
\begin{alignat}{4}
\curl' h' &= j'_s \quad &\text{in } \Omega', \qquad && h' &= \partial_{b'} w'(b'), \label{eq:1h}\\
\div' b' &= 0 \quad & \text{in } \Omega', \qquad &&  b' \cdot n' &= 0 \quad \text{on } \partial\Omega'. \label{eq:2h}
\end{alignat}
We again assume that $j_s' = \curl' h_s'$ for some $h_s' \in H(\curl';\Omega')$. 
We then use the pull-back principle~\cite{Boffi2013,Ciarlet2002} to transform this system into an equivalent problem on the reference domain $\Omega$. To do so, we use the notation $x'=\phi(x)$, $F(x)=D\phi(x)$, $J(x)=\operatorname{det}(F(x))$ and define $b(x)$, $h(x)$, $h_s(x)$, and $w(b)$ through
\begin{alignat*}{5}
b'(x') &= \tfrac{1}{J(x)} F(x) b(x), \qquad &
h'(x') &= F(x)^{-\top} h(x), \\
h_s'(x') &= F(x)^{-\top} h_s(x), \qquad &
w'(x',b'(x')) &= \tfrac{1}{J(x)} w(x,b(x)).
\end{alignat*}
With these definitions and the transformation rules for the differential operators and the normal vector \cite{Arnold2019,Boffi2013}, one can see that the problem \eqref{eq:1h}--\eqref{eq:2h} posed on the physical domain $\Omega'$ is equivalent to the system \eqref{eq:1}--\eqref{eq:2} on the reference domain~$\Omega$. 
From the above formulas, one can further observe that 
\begin{align*}
\partial_{b} w(x,b) 
= J(x) \partial_b w'(x', \tfrac{1}{J(x)} F(x) b) 
=  F(x)^\top \partial_{b'} w'(x',b'),    
\end{align*}
which follows from the above relations between the functions and the chain rule of differentiation. This is in perfect agreement with $h(x) = F(x)^\top h'(x')$ and the above formulae.
As a consequence, one can see that 
\begin{align*}
\langle \partial_b w(b_1) - \partial_b w(b_2), b_1-b_2\rangle_\Omega
&= \langle \partial_{b'} w'(b_1') - \partial_{b'} w'(b_2'), b_1' - b_2'\rangle_{\Omega'},
\end{align*}
which implies the strong monotonicity of $\partial_b w(b)$; Lipschitz continuity follows in the same way.  
Also, the regularity requirements on the fields $h_s$, $h$, and $b$ made in Theorem~\ref{theorem:discretization:error} translate verbatim to corresponding piecewise regularity conditions on the fields $h_s'$, $h'$, $b'$ on the physical domain $\Omega'$. 
Hence all results of the previous sections transfer to problems on curved domains $\Omega' = \phi(\Omega)$. 

\begin{remark}
The mapping trick utilized in the arguments above was proposed and intensively used for the study of partial differential equations on surfaces; see e.g. \cite{Deckelnick2005,Deckelnick2009}. 
Further note that the material behavior on the reference domain $\Omega$ will in general be anisotropic, even if the physical behavior on $\Omega'$ was isotropic. This can be seen as a good reason to consider anisotropic material laws right from the beginning.
By the same transformation formulas, the method considered in Sections~\ref{sec:discretization}--\ref{sec:quadratic} could be phrased equivalently on the curved domain $\Omega' = \phi(\Omega)$, which is the usual computational practice. By equivalence, all assertions of the previous sections carry over verbatim.
\end{remark}

\subsection{Two space dimensions}
\label{sec:2d}

In the context of electric machine simulation, the following two-dimensional setting is of importance~\cite{heise94,Meunier2008}. 
We consider 
\begin{alignat}{4}
\curl h &= j_s \quad &\text{in } \Omega, 
\qquad && 
h &= \denergy(b), \label{eq:1_2d}\\
\div b &= 0 \quad & \text{in } \Omega, 
\qquad && 
b \cdot n &= 0 \quad \text{on } \partial\Omega, \label{eq:2_2d}
\end{alignat}
where now $\Omega \subset \RR^2$ is a two-dimensional domain, e.g., the cross-section of an electric motor, and $h$, $b:\Omega \to \RR^2$ are the in-plane components of the magnetic field and flux, and $j_s : \Omega \to \RR$ is the out-of-plane component of the driving current. 
The differential operators are defined by $\curl h = \partial_x h_2 - \partial_y h_1$ and $\div b = \partial_x b_1 + \partial_x b_2$. 
Under the assumption that $\Omega$ is topologically trivial, we can represent $b=\Curl a$ with $a : \Omega \to \RR$ corresponding to the out-of-plane component of the magnetic vector potential and $\Curl a = (\partial_y a, -\partial_x a)$ denoting the scalar--to--vector $\curl$--operator.  
As solution space for the weak formulation of the above problem, we choose
\begin{align*}
V_0 = H_0(\Curl) = \{u \in L^2(\Omega) : \Curl u \in L^2(\Omega)^2 \ \text{and} \ u|_{\partial\Omega}=0\},
\end{align*}
which is equipped with the graph norm $\|u\|_{H(\Curl)} = (\|u\|_{L^2}^2 + \|\Curl u\|_{L^2}^2)^{1/2}$.
We note that $\|\Curl u\|_{L^2(\Omega)} \simeq \|u\|_{H(\Curl)}$ defines a norm. 
All results of the previous sections then translate verbatim to the two-dimensional setting. 

\begin{remark}
We note that $\Curl a = (\dy a, -\dx a) 
= (\dx a,\dy a)^\perp 
= (\nabla a)^\perp$, and hence $H_0(\Curl;\Omega)$ is often  identified with $H_0^1(\Omega)$; see e.g. \cite{heise94}. 
For isotropic materials, i.e., a scalar--valued reluctivity $\nu$, one can then write 
\begin{align}
\langle \nu \Curl a, \Curl a'\rangle_\Omega 
= \langle \nu \nabla a, \nabla a'\rangle_\Omega,
\end{align}
which is frequently done in the literature. Such an identification, however, introduces additional complications in the case of anisotropic materials, and we advocate to stay with the natural function space $H_0(\Curl;\Omega)$ instead.
\end{remark}

\section{Numerical tests}
\label{sec:numerics}
To illustrate our theoretical results and to demonstrate the efficiency of the proposed numerical schemes, let us briefly report on some computational tests. 

\subsection{Details on the implementation}
In our computations, we use Nedelec elements $\N_k$ on tetrahedral meshes in three dimensions. For the two-dimensional setting, we use Courant elements $P_{k+1}$ on triangular meshes.
The grids are adapted to the curved boundaries of the geometries where needed. 
A quadrature rule of order $2k$ is applied on each element for integration in accordance with the conditions of Assumption~\ref{ass:2} and Theorem~\ref{theorem:discretization:error}.
%
The nonlinear systems arising after discretization are solved by the Newton method \eqref{eq:opt1}--\eqref{eq:opt2} with initial iterate $a^0=0$ and step size $\tau^n$ determined by Armijo backtracking~\eqref{eq:opt3}.
Tree-cotree gauging~\cite{Albanese1988,Zaglmayr2005} is used to ensure uniqueness of the vector potential in three space dimensions. 
The symmetric and positive definite linear systems arising in every Newton step are solved by the conjugate gradient method. 
The error tolerances in the linear and nonlinear solver are chosen sufficiently small so not to influence the convergence behavior. 
The three-dimensional computations were realized in the finite element package \texttt{Netgen/NGSolve}~\cite{Schoeberl2014} and \texttt{ParaView}~\cite{paraview} was used for visualization.
The two-dimensional results were obtained and visualized in \texttt{Matlab}.

\subsection{Smooth solution} 
As a first test case, we choose a problem for which a smooth solution can be expected. The geometry consists of a cylinder 
\begin{align*}
\Omega = \{(x,y,z) : x^2+y^2 < R^2, \ 0 < z < H\}
\end{align*}
of radius $R=100 \, \text{mm}$ and height $H=12 \, \text{mm}$. 
A current of density $j = \pm I e_z$, $I=10^5 \, \text{A}/\text{m}^3$, $e_z=(0,0,1)$, is conducted through two cylindrical wires 
\begin{align*}
\Omega_{\pm} = \{(x,y,z) : x^2 + (y \mp \ell)^2 < r^2, \ 0<z<H\}
\end{align*}
with $\ell=50 \, \text{mm}$ and $r=25\, \text{mm}$, filled by copper. The constitutive relation on these domains is defined as $h = \nu_0 b$ with $\nu_0 = \frac{1}{4\pi} 10^7 \, \text{m}/\text{H}$, which amounts to an energy density $w(b) = \frac{\nu_0}{2} |b|^2$.
The remainder of the domain $\Omega_I = \Omega \setminus (\Omega_+ \cup \Omega_-)$ is filled with iron, whose magnetic properties are described by a modified Brauer model~\cite{Brauer1975}. The corresponding energy density is given by $w(b) = \widetilde w(|b|)$ with 
\begin{align}
\widetilde w(s) = 
\begin{cases} 
\frac{k_1}{2 k_2} e^{k_2 s^2} + \frac{k_3}{2} s^2, & s \le s_* \\
a_0 + a_1 s + \frac{\nu_0}{2} s^2, & s > s_*
\end{cases}
\end{align}
with $k_1 = 3.8$, $k_2 = 2.17$, and $k_3 = 396.2$ like suggested in \cite{Brauer1975}. The parameters $a_i$ and the threshold $s_* \approx 2.06 \, \text{T}$ are chosen to ensure $C^2$ continuity of~$\widetilde w$. By construction, the function $w(b)=\widetilde w(|b|)$ then satisfies the conditions of Assumption \ref{ass:1}.
The geometric setup of our test problem is illustrated in the left plot of Figure~\ref{fig:smooth3d}. 
On the top and bottom surface, we prescribe symmetry boundary conditions $h \times n = 0$, and on the lateral boundaries, we set $b \cdot n = 0$. 

In Figure~\ref{fig:smooth3d}, we depict the $b$ and $h$ fields on the surface of the domain $\Omega$. %
\begin{figure}[ht!]
    \centering
    \includegraphics[width=0.32\linewidth]{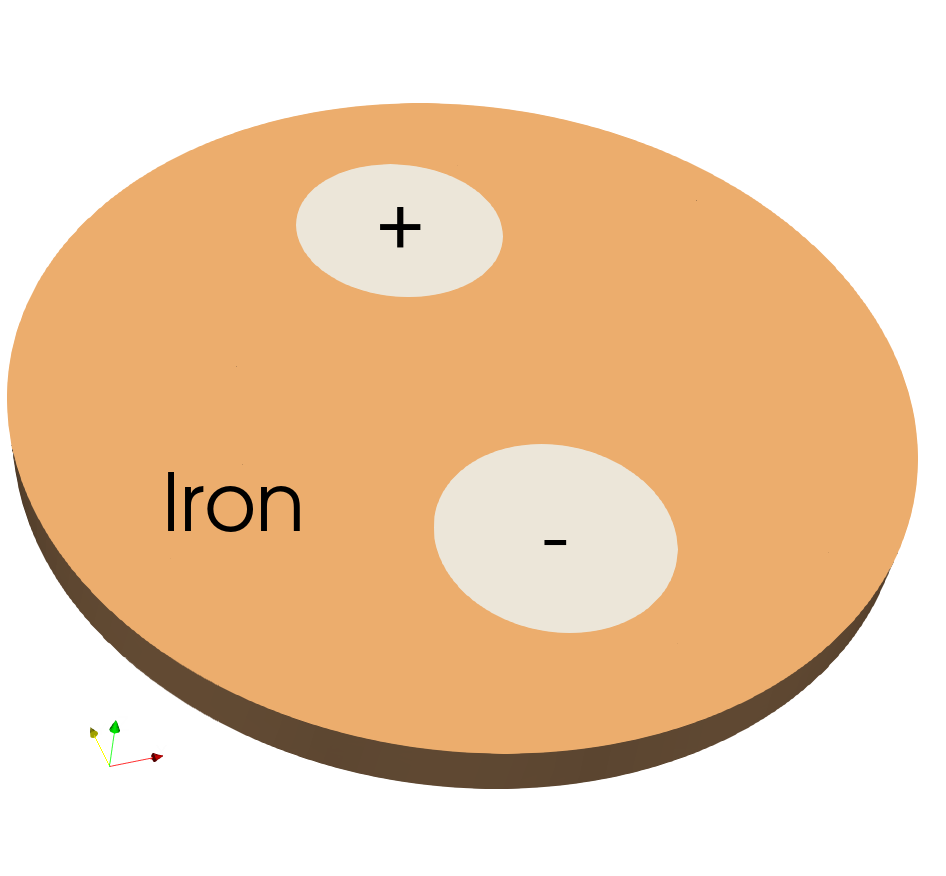}
    \includegraphics[width=0.32\linewidth]{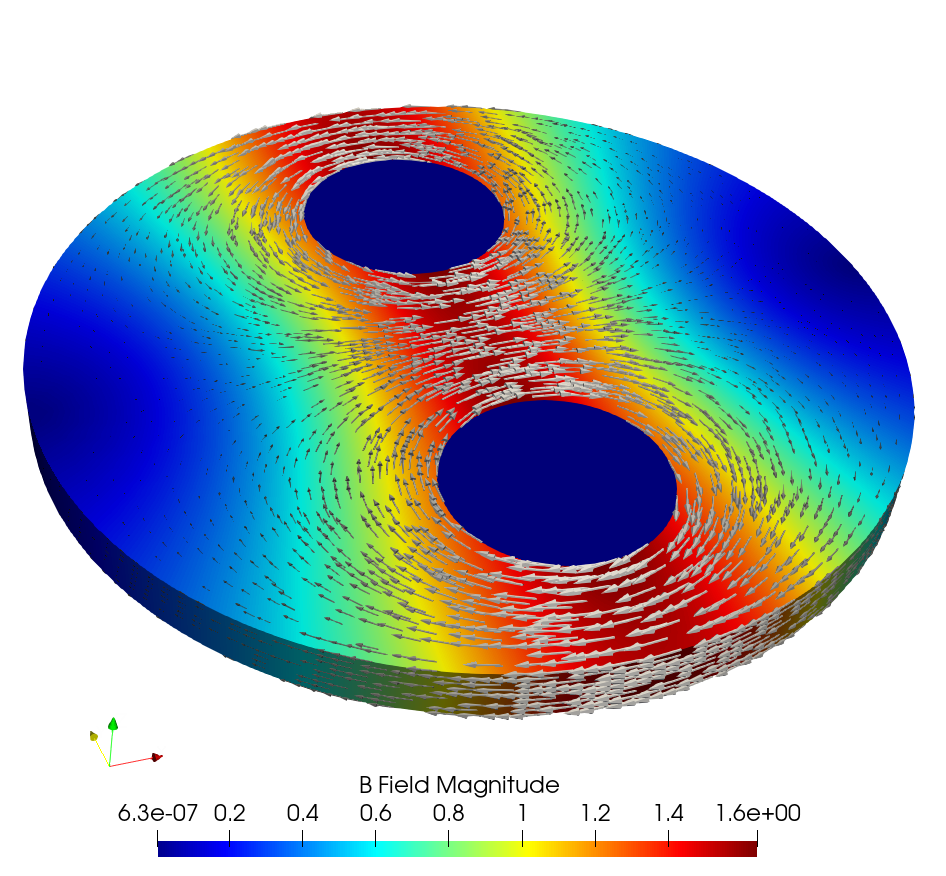}
    \includegraphics[width=0.32\linewidth]{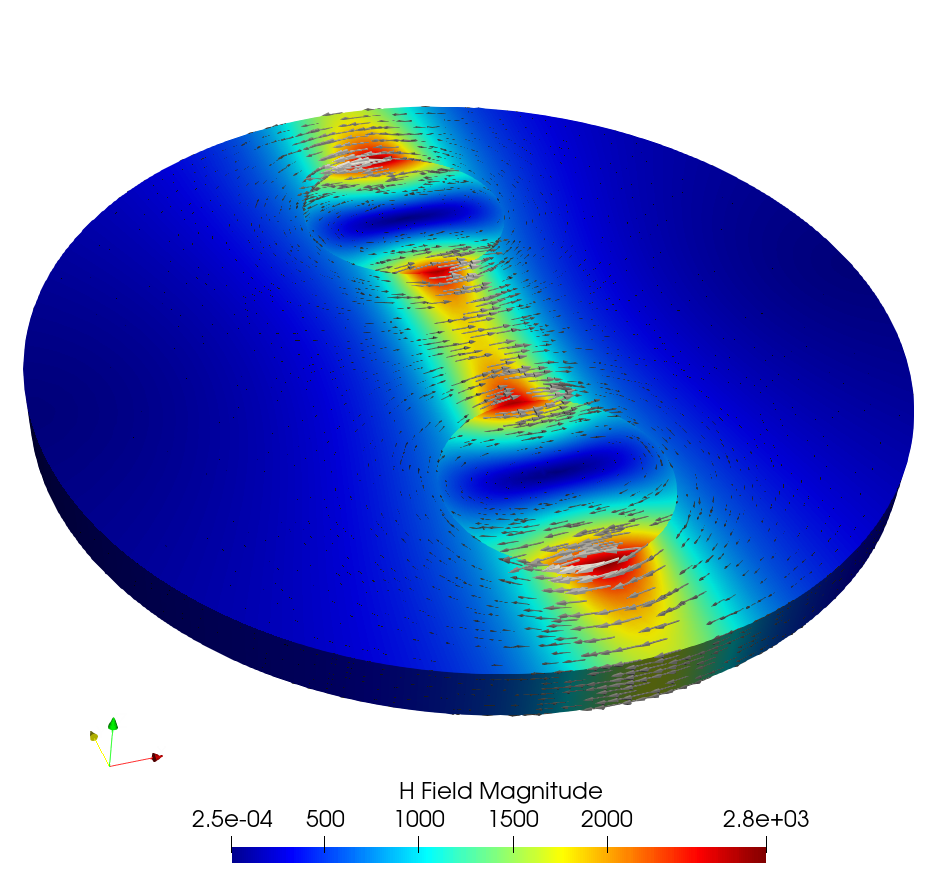}
    \caption{Geometry sketch (left), the magnitude of the $B$-field (middle) and the magnitude of the $H$-field}
    \label{fig:smooth3d}
\end{figure}%
As can be seen from the plots, the magnitude $|b|$ of the magnetic flux does not exceed the threshold value $s_*$. In particular, the solution is piecewise smooth and does not exhibit singularities. As a consequence, we expect to obtain the optimal convergence rates predicted by Theorem~\ref{theorem:discretization:error} and Remark~\ref{remark:convergence:b:h}.
Since no exact solution is available for this test problem, we use finite element approximations $b_h^k = \curl a_h^k$, $h_h^k=\partial_b w(b_h^k)$ obtained from vector potentials $a_h^k \in \N_k$ for different polynomial degree~$k$ to estimate the errors. 
In Table~\ref{tab:ex1a} and \ref{tab:ex1b}, we summarize the results of our computations. %
\begin{table}[ht!]
\centering
\setlength\tabcolsep{1.5ex}
\renewcommand{\arraystretch}{1.1}
\begin{tabular}{r|r|c||c|c||c|c} 
ne & dof & iter& $\frac{\|b_{h}^1 - b_h^2\|_{L^2}}{\|b_h^2\|_{L^2}}$ 
    & eoc$_b$ & $\frac{\|h_{h}^1 - h_h^2\|_{L^2}}{\|h_h^2\|_{L^2}}$  & eoc$_h$  \\
\hline
\hline
$629$    & $4.057$     & $9$ & $0.017504$ & $-$    & $0.053846$ & $-$  \\
$5.032$   & $30.289$    & $9$ & $0.005686$ & $1.62$ & $0.019633$ & $1.46$ \\
$40.256$  & $229.097$   & $9$ & $0.001603$ & $1.82$ & $0.005643$ & $1.80$  \\
$322.048$ & $1.775.673$ & $9$ & $0.000414$ & $1.95$ & $0.001501$ & $1.91$ 
\end{tabular}
\caption{Number of elements (ne), degrees of freedom (dof), Newton iterations (iter), errors, and the estimated order of convergence (eoc) for the $b$ and $h$ field using an approximation for the vector potentials $a$ in $\N_k$, $k=1$.}
\label{tab:ex1a}
\end{table}
\begin{table}[ht!]
\centering
\setlength\tabcolsep{1.5ex}
\renewcommand{\arraystretch}{1.1}
\begin{tabular}{r|r|c||c|c||c|c} 
ne & dof & iter& $\frac{\|b_{h}^2 - b_h^3\|_{L^2}}{\|b_h^3\|_{L^2}}$ 
    & eoc$_b$ & $\frac{\|h_{h}^2 - h_h^3\|_{L^2}}{\|h_h^3\|_{L^2}}$  & eoc$_h$  \\
\hline
\hline
$629$    & $10.399$    & $9$ & $0.002096$ & $-$ & $0.007777$ & $-$  \\
$5.032$   & $78.301$    & $9$ & $0.000412$ & $2.35$ & $0.001772$ & $2.13$ \\
$40.256$  & $602.297$   & $9$ & $0.000069$ & $2.57$ & $0.000293$ & $2.59$  \\
$322.048$ & $4.717.689$ & $9$ & $0.000009$ & $2.87$ & $0.000042$ & $2.78$ 
\end{tabular}
%
\caption{Number of elements (ne), degrees of freedom (dof), Newton iterations (iter), errors, and the estimated order of convergence (eoc) for the $b$ and $h$ field using an approximation for the vector potentials $a$ in $\N_k$, $k=2$.}
\label{tab:ex1b}
\end{table}%
The numerical results also clearly demonstrate that the convergence behavior of the Newton method is independent of the mesh size $h$ and the polynomial degree $k$; compare with Remark \ref{remark:meshindependent}.

Let us finally note that the specific setup of the model problem complies with the two-dimensional setting discussed in Section~\ref{sec:2d}. Similar results could thus also be obtained here by computations in two dimensions. 
%

\subsection{TEAM Problem 13}
Our second example is motivated by one of the benchmark problems of the COMPUMAG TEAM suite \cite{teamproblem13}.
The geometry here consists of a coil made of copper, surrounded by a specific assembly of iron plates. These components are placed within an air box. 
No flux boundary conditions $b \cdot n=0$ are prescribed at the outer boundary and a total current of $3000 \, \text{A}$ is assumed to flow through the vertical cross sections of the coil. The corresponding current density $j$ is determined by solving a Poisson problem in a pre-processing step. 
A sketch of the geometry is depicted in the left part of Figure~\ref{fig:team13}; see \cite{teamproblem13} for details. 
The magnetic behavior of the regions filled by air and copper is described by the linear relation $h = \nu_0 b$, and the modified Brauer model presented in the previous example is used for the ferromagnetic plates. 
The distribution of the current and magnetic flux density is depicted in the right plot of Figure~\ref{fig:team13}.
\begin{figure}[ht!]
    \centering
    \includegraphics[width=0.45\linewidth]{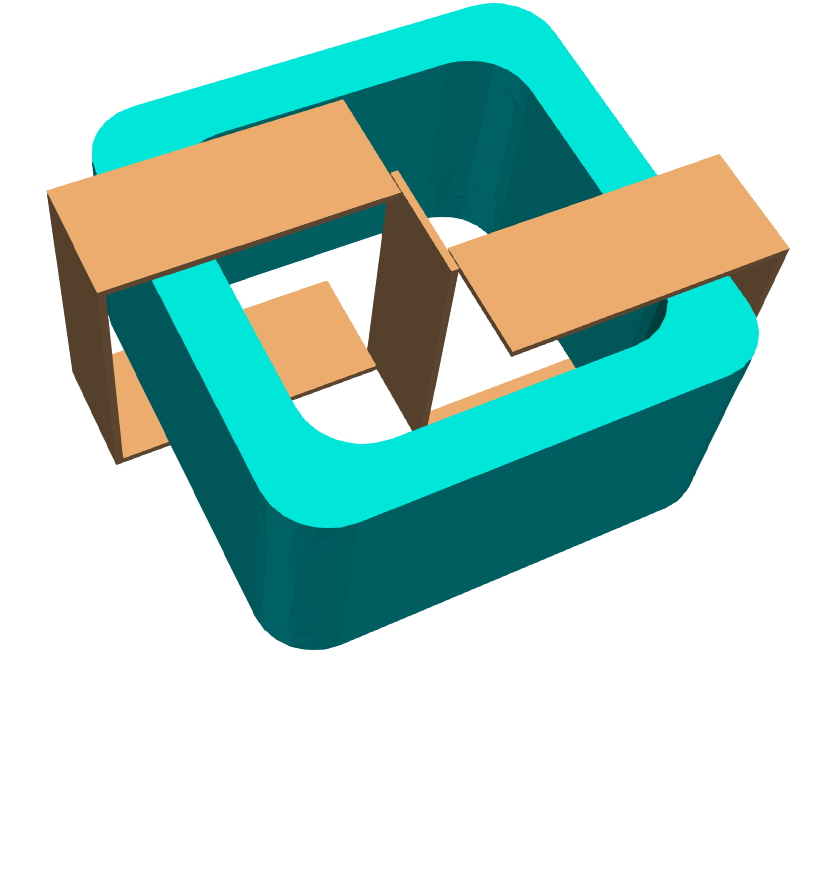}
    \includegraphics[width=0.45\linewidth]{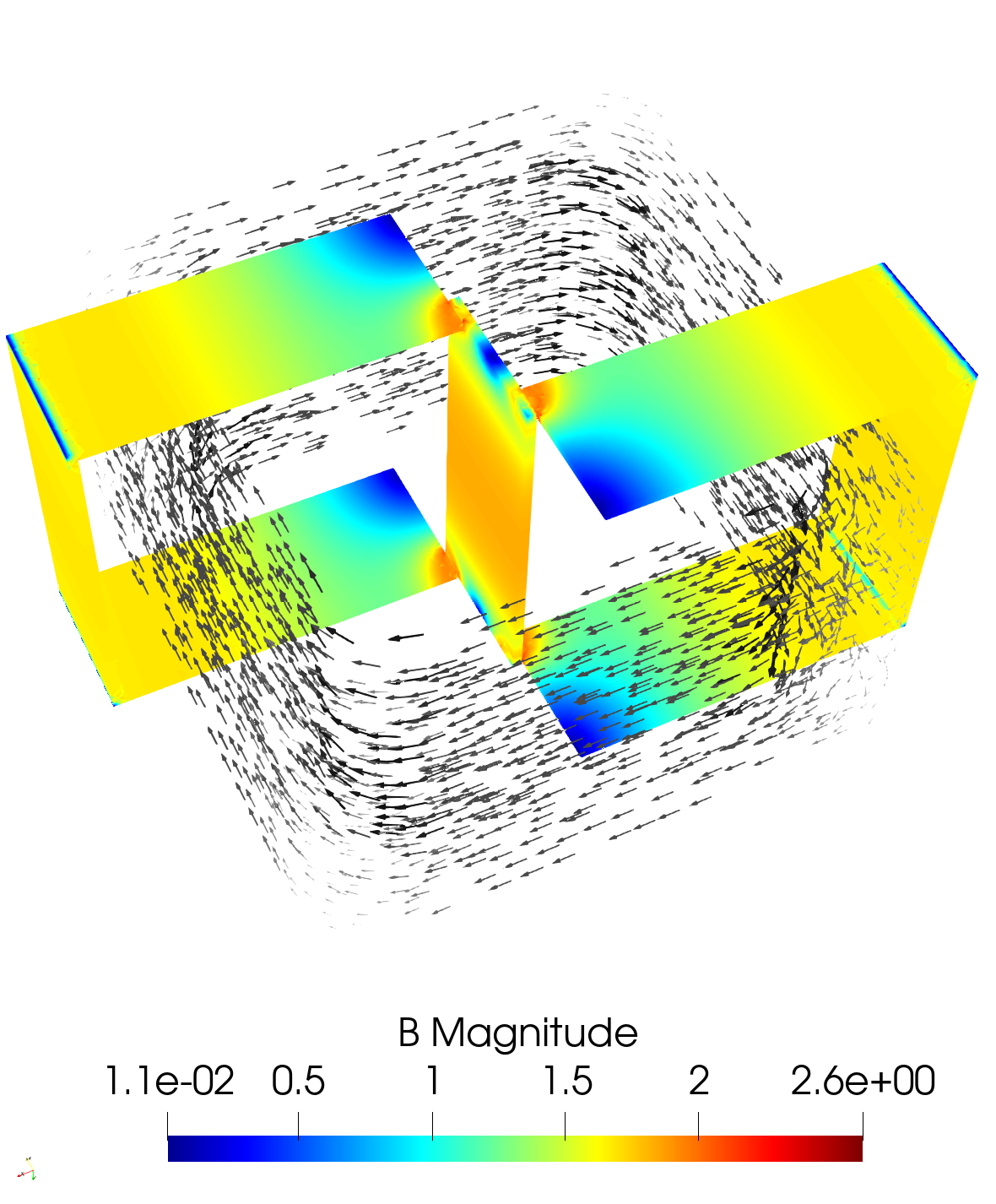}
    \caption{Left: Geometric setup of TEAM problem 13; compare with \cite{teamproblem13}. Right: Current density $j$ in the coil (black arrows) and magnitude of the magnetic flux $|b|$ in the iron plates (color). }
    \label{fig:team13}
\end{figure}
Due to the discontinuities in the material laws, the exact solution of this problem exhibits singularities at edges and corners of the iron plates, which are clearly visible in the field plots. 
\begin{table}[ht!]
\centering
\setlength\tabcolsep{1.5ex}
\renewcommand{\arraystretch}{1.1}
\begin{tabular}{r|r|c||c|c||c|c} 
ne & dof & iter& $\frac{\|b_{h}^1 - b_h^2\|_{L^2}}{\|b_h^2\|_{L^2}}$ 
    & eoc$_b$ & $\frac{\|h_{h}^1 - h_h^2\|_{L^2}}{\|h_h^2\|_{L^2}}$  & eoc$_h$  \\
\hline
\hline
$21.087$    & $109.869$ & $11$ & $0.040482$ & $-$   & $0.090883$ & $-$   \\
$168.696$   & $900.106$ & $11$ & $0.021661$ & $0.90$ & $0.053695$ & $0.76$ \\
$1.349.568$   & $7.209.588$ & $11$ & $0.012143$ & $0.83$ & $0.031389$ & $0.77$ 
\end{tabular}
\caption{Number of elements (ne), degrees of freedom (dof), Newton iterations (iter), errors, and estimated order of convergence (eoc) for the $b$ and $h$ field using approximations for the vector potential $a$ in $\N_k$, $k=1$}
\label{table:team13conv}
\end{table}
As a consequence, only limited convergence rates can be expected on uniformly refined meshes. 
The corresponding results of our computations are summarized in  Table~\ref{table:team13conv}. 
As predicted by Theorem~\ref{thm:global}, the convergence behavior of the Newton iteration is, however, not affected by the singularities of the solution.
Here the truncation used in the definition of the magnetic energy density $\widetilde w(s)$ becomes effective.

\subsection{Permanent magnetic synchronous machine}
As a last test case, we consider a typical problem arising in electric machine simulation. The geometry here represents a cross-section of a permanent magnet synchronous machine consisting of a stator and a rotor separated by a small air gap; see Figure~\ref{fig:PMSM} for a sketch.
\begin{figure}[ht!]
    \centering
    \raisebox{0.39\height}{\includegraphics[trim={3.0cm 1.07cm 2.5cm 0.7cm},clip,width=0.28\textwidth]{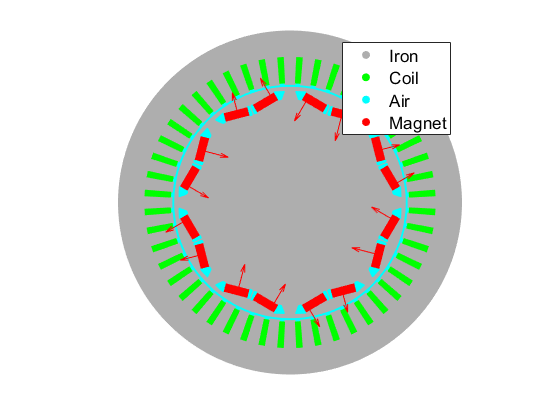}}
    \hspace{0.1cm}
    \includegraphics[trim={3cm 0cm 3cm 0cm},clip,width=0.33\textwidth]{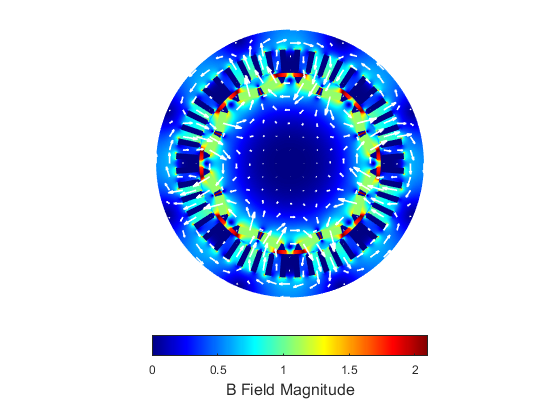}
    \includegraphics[trim={3cm 0cm 3cm 0cm},clip,width=0.33\textwidth]{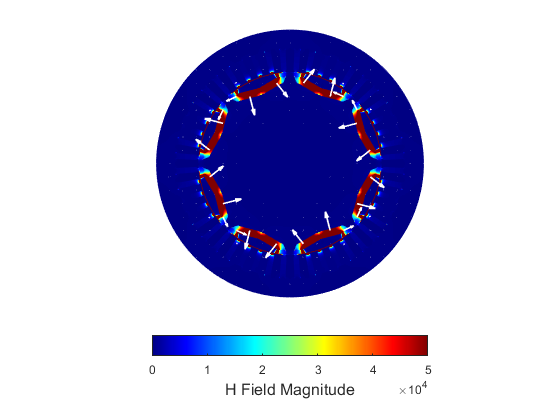}
 \caption{Left: geometric setup of PMSM with iron (grey), magnet (red), coil (green) and air (cyan). The orientation of the magnets is depicted with red arrows. Middle: magnitude of the magnetic flux $|b| = |\curl a|$. Right: magnitude of the magnetic field $|h| = |\partial_b w(b)|$.}
 \label{fig:PMSM}
\end{figure}
The stator yoke and teeth as well as the rotor core are made up of iron, while the stator windings are filled with copper.  Permanent magnets are placed in the rotor and partially surrounded by air pockets
A detailed  description of the geometric setup can be found in \cite{Gangl2024}. 

The ferromagnetic behavior of the iron is described by the modified Brauer model already used in the previous simulations. For air and copper regions, we use the linear relation $h = \nu_0 b$. 
The permanent magnets are modeled by the relation $h = \nu_0 b - m$ with magnetization vector $m$ prescribed individually for each magnet; see the left plot of Figure~\ref{fig:PMSM} and \cite{Gangl2024} for details. 
The two-dimensional setting discussed in Section~\ref{sec:2d} is employed for our simulations.
No-flux conditions $b \cdot n = 0$ are chosen at the outer boundary and the current density $j_s$ in the windings is set to zero.  
In our simulations, we thus only consider the magnetic field generated by the permanent magnets. 
The corresponding solutions are depicted in Figure~\ref{fig:PMSM}. 
Since the exact solution is again not known for this example, we use numerical solutions computed on different grids to estimate the finite element errors. 
The results of our computations are summarized in Table~\ref{tab:PMSM}. 
\begin{table}[ht!]
\centering
\setlength\tabcolsep{1.5ex}
\renewcommand{\arraystretch}{1.1}
\begin{tabular}{c||r|c||c|c||c|c} 
ne & dof & iter & $\frac{\|b_{h/2} - b_h\|_{L^2}}{\|b_{h/2}\|_{L^2}}$ 
    & eoc$_b$ & $\frac{\|h_{h/2} - h_h\|_{L^2}}{\|h_{h/2}\|_{L^2}}$ & eoc$_h$  \\
\hline
\hline
$4424$ & $8793$    & $16$    & $0.089158$   & $-$  & $0.153982$  & $-$  \\
$17696$ & $35281$   & $18$    & $0.064944$   & $0.45$  & $0.104631$  & $0.55$ \\
$70784$ & $141345$  & $17$    & $0.037391$  & $0.79$  & $0.056812$  & $0.88$  \\
$283136$ & $565825$  & $17$    & $0.018431$  & $1.02$     & $0.030357$  & $0.90$
\end{tabular}
\caption{Number of elements (ne), degrees of freedom (dof), Newton iterations (iter), errors, and estimated order of convergence (eoc) for the $b$ and $h$ field using an approximation for the vector potential $a$ in $P_{k+1}$, $k=1$.}
\label{tab:PMSM}
\end{table}
Similarly to the previous example, the solution exhibits singularities at geometric corners, where the material laws are discontinuous. As a consequence, we cannot expect full convergence rates on uniformly refined meshes.
The iteration numbers for the Newton method are again independent of the mesh size and not affected by the singularities of the solution.

\section{Summary and Discussion}

In this paper, we studied the systematic approximation of nonlinear problems in magnetostatics by higher-order finite element methods. 
Numerical quadrature was considered for treating the nonlinearities and global mesh-independent convergence of a damped Newton method was established. 
The consideration of an energy-based anisotropic material law was a key in the error and convergence analysis, and further allowed a seamless generalization of the results to problems domains with curved boundaries. 
The feasibility of the proposed methods and validity of the theoretical results were demonstrated with computational test for some typical benchmark problems. 
While the contraction factors of the Newton method were not affected by singularities of the solution, the convergence rates of the errors were of course limited.  
Adaptive and anisotropic mesh refinement would be required to reveal the full order of convergence also in these examples. 
The consideration of adaptive mesh refinement and a-posteriori error estimation as well as the consideration of higher-order approximations for alternative formulations of nonlinear magnetostatics are left as topics for future research.

\bigskip 

\begingroup
\small 
\subsection*{Acknowledgements}
This work was  supported by the joint DFG/FWF Collaborative Research Centre CREATOR (DFG: Project-ID 492661287/TRR 361; FWF: 10.55776/F90) at TU Darmstadt, TU Graz,  JKU Linz, and RICAM.
\endgroup

\bigskip 


\end{document}